\documentclass[a4paper,12pt]{article}
\usepackage{amssymb}
\usepackage{latexsym}
\usepackage{amsmath,amsthm,amssymb}
\usepackage{amsfonts}

\usepackage{graphicx}
\usepackage{float}

\newtheorem{thm}{Theorem}[section]

\newtheorem{pro}[thm]{Proposition}

\theoremstyle{definition}
\newtheorem{defn}[thm]{Definition}
\newtheorem{exa}[thm]{Example}
\newtheorem{rem}[thm]{Remark}

\begin{document}

\begin{center}
{\Large Multivariate $\alpha$-normal distributions}
\end{center}

\begin{center}
{\sc Krzysztof Zajkowski}\\
Faculty of Mathematics, University of Bialystok \\ 
Ciolkowskiego 1M, 15-245 Bialystok, Poland \\ 
kryza@math.uwb.edu.pl 
\end{center}

\begin{abstract}
The Weibull distribution can be obtained using a power transformation from the standard exponential distribution. In this article, we will consider a symmetrized power transformation of a random variable with the standard normal distribution. We will call its distribution the $\alpha$-{\it normal (Gaussian) distribution}. We examine  properties of this distribution in detail.
We calculate moments and consider the moment problem of $\alpha$-normal distribution. We derive the formula of its differential entropy and (exponential) Orlicz norm. 
Moreover, we define the joint distribution function of the multivariate 
$\alpha$-normal distribution as a meta-Gaussian distribution with 
$\alpha$-normal marginals. We consider also the limiting distribution as $\alpha$ tends to infinity.
\end{abstract}


\section{Introduction}
The Weibull distribution is one of the most important and well-known probability distributions, which has a wide range of applications (see, e.g., Rinne's extensive monograph \cite{Rin}).
The Weibull distribution can be obtained from the standard exponential distribution by power transformation
(see Johnson, Kotz and Balakrishnan \cite[Ch.8]{JKBal}), i.e., if $X$ has the exponential distribution with  parameter $1$ then 
$W_{\alpha,\lambda}:=\lambda X^{1/\alpha}$  ($\alpha,\lambda>0$) has 
two-parameter Weibull distribution with the shape parameter $\alpha$ and the scale parameter $\lambda$.
The random variable (r.v.) $W_{\alpha,\lambda}$ has $\alpha$-exponential tail decay and belongs to the corresponding exponential 
type Orlicz space (see Section \ref{secOrlicz}).

In this article, we propose applying a power transformation of the form $\operatorname{sgn}(x)|x|^{2/\alpha}$ to a random variable $G$ with the standard normal distribution to obtain some equivalent to the Weibull random variable. This variable will have the same order of tail decay and will belong to the same Orlicz space as the variable $W_{\alpha,\lambda}$. 

A multivariate Weibull distribution can in general be obtained from multivariate exponential distribution by power transformations of marginals (see Kotz, Balakrishnan and Johnson \cite[Ch. 47.4]{KBJ}, i.e., if 
$X_1,...,X_n$ have some joint exponential distribution then the power transformations 
$X_1^{1/\alpha_1},...,X_n^{1/\alpha_n}$ $(\alpha_1,...,\alpha_n>0)$  produce a multivariate distribution having Weibull marginals. There are many forms of multivariate exponential distributions. In general it depends on
copulas used (see \cite[Ch. 47]{KBJ}). In this paper, we apply the particular Gauss copula to the $\alpha$-normal marginals. We  call this probability distribution a {\it multivariate $\alpha$-normal distribution}. In other words, we define this distribution as a meta-Gaussian distribution with $\alpha$-normal marginals.
We consider also the limiting distribution as $\alpha$ tends to infinity.

Meta-Gaussian distributions with different marginals have various applications, for example, in financial mathematics (see  \cite{Li}), information theory (see \cite{lul} ), 
hydrology (see \cite{Rey}). In this paper, we propose a new notion of the meta-Gaussian distribution with 
$\alpha$-normal marginals.

\section{The $\alpha$-normal distribution}

We define a power transformation of the normal distribution, which  was announced in \cite{Zaj}.
\begin{defn}
Let $G$ be the standard normal distributed random variable and $\alpha$ be a positive number.
We denote by $G_\alpha$  the random variable $\operatorname{sgn}(G)|G|^{2/\alpha}$, where $\operatorname{sgn}(x)$ is the signum function of $x$. 
We call $G_\alpha$  the {\it model $\alpha$-normal ($\alpha$-{\it Gaussian}) random variable}.
\end{defn}
\begin{rem}
Let us emphasize that $G_2=\operatorname{sgn}(G)|G|=G$. 
$G_2$ is the standard normally distributed random variable. 
All results that we  obtain for $G_\alpha$ 
are generalizations known facts for the  normal distribution.
\end{rem}
\begin{rem}
Since the function $f_\alpha(x):=\operatorname{sgn}(x)|x|^{2/\alpha}$ is odd and $G$ is a symmetric random variable, $G_\alpha=f_\alpha(G)$ is the symmetric one and $|G_\alpha|$ has the same distribution 
as $|G|^{2/\alpha}$.
\end{rem}
\begin{rem}
Tails of the Gaussian random variable can be estimated from above in the following way
$$
\mathbb{P}(|G|\ge x)\le \exp(-x^2/2)
$$
for any $x\ge 0$; see, for instance, \cite[Prop.2.2.1]{Dudley}). Hence for $\alpha$-normal random variable we get
$$
\mathbb{P}(|G_\alpha|\ge x)=\mathbb{P}(|G|^{2/\alpha}\ge x)=\mathbb{P}(|G|\ge x^{\alpha/2})\le \exp(-x^\alpha/2).
$$
This means that $G_\alpha$ has 
the \emph{$\alpha$-sub-exponential tails decay}. 
\end{rem}
\begin{pro}
\label{cdf}
i) The distribution function (d.f.) $\Phi_\alpha$ of $G_\alpha$ is of the form
$$
\Phi_\alpha(x)=\Phi(\operatorname{sgn}(x)|x|^{\alpha/2}),
$$
where $\Phi$ is the standard normal distribution function.\\
ii) The probability density function $\varphi_\alpha$ of $G_\alpha$ is of the form
$$
\varphi_\alpha(x)=\frac{\alpha}{2\sqrt{2\pi}}|x|^{\alpha/2-1}\exp(-|x|^{\alpha}/2).
$$
\end{pro}
\begin{proof}
Let us observe that $f_\alpha:\mathbb{R}\mapsto\mathbb{R}$ is bijection and its inverse  is of the form
$f_\alpha^{-1}(x)=\operatorname{sgn}(x)|x|^{\alpha/2}$. Thus 
the distribution function  of the $\alpha$-Gaussian random variable $G_\alpha$, which we will denote by $\Phi_\alpha$, has the form
\begin{eqnarray*} 
\Phi_\alpha(x)&=&\mathbb{P}(G_{\alpha}\leq x)=\mathbb{P}(\operatorname{sgn}(G)|G|^{2/\alpha}\leq x)\\
\; &=& \mathbb{P}(G\leq \operatorname{sgn}(x)|x|^{\alpha/2})=\Phi(\operatorname{sgn}(x)|x|^{\alpha/2}).   
\end{eqnarray*}
For the second part, since 
$$
\frac{d}{dx}\big(\operatorname{sgn}(x)|x|^{\alpha/2}\big)=\frac{\alpha}{2}|x|^{\frac{\alpha}{2}-1},
$$
we infer that
the density function $\varphi_\alpha=\Phi^\prime_\alpha$ has the form as in  the assertion.
\end{proof}

{\bf Density function description} (The following description and the graphic were made by the student Jacek Oszczepali\'nski).


The density function of the random variable $G_{\alpha}$ is an even function. If we consider $x > 0$, the derivative of this function has the following expression
\begin{eqnarray*}
\varphi_{\alpha}'(x)
&=&\frac{-\alpha^2}{4\sqrt{2\pi}}x^{\frac{\alpha}{2}-2}\Big(x^{\alpha}-\frac{\alpha-2}{\alpha}\Big)\exp\Big(-\frac{1}{2}x^{\alpha}\Big).
\end{eqnarray*}

It is worth noting that for $0 < \alpha < 2$, the density function exhibits an infinite negative slope at $0$ (i.e., $\lim\limits_{x\to 0^+} \varphi'_{\alpha}(x) = -\infty$), and it is negative for all $x > 0$. In the case of $\alpha = 2$, the slope at $0$ is finite, and we have $\lim\limits_{x\to 0^+} \varphi_{2}'(x) = \varphi_{2}'(0) = 0$. For $2 < \alpha < 4$, the slope at $0$ is infinitely positive, and when $\alpha = 4$, we have $\lim\limits_{x\to 0^+} \varphi_{4}'(x) = 2/\sqrt{\pi}$, while $\varphi'_{\alpha}(0) = 0$ for $\alpha > 4$. In general, for $2 < \alpha$ and $x > 0$, the slope of the density function is positive up to the value $\sqrt[\alpha]{\frac{\alpha-2}{\alpha}}$, where $\varphi'_{\alpha}\Big(\sqrt[\alpha]{\frac{\alpha-2}{\alpha}}\Big)=0$, and it is negative beyond the aforementioned threshold.\\

The shape of the density function of the standard $\alpha$-normal distribution undergoes a significant transformation depending on the value of $\alpha$. Specifically, for values within the range of $0 < \alpha < 2$, the function exhibits a vertical asymptote at zero. When $\alpha$ equals $2$, the function describes the density of a standard normal distribution. However, for $\alpha > 2$, the function features a local minimum at zero with a value of zero and two maxima at $\pm\sqrt[\alpha]{\frac{\alpha-2}{\alpha}}$, depicting a distinct bimodal distribution.\\

The figure below displays graphs of $\alpha$-normal probability density functions, each with different values of the shape parameter $\alpha$. Specifically, the density functions represented by the colours red, blue, purple, and green correspond to $\alpha$ values of $1$, $2$, $3$, and $5$, respectively.\newline

\begin{figure}[H] 
\centering
\includegraphics[width=12cm]{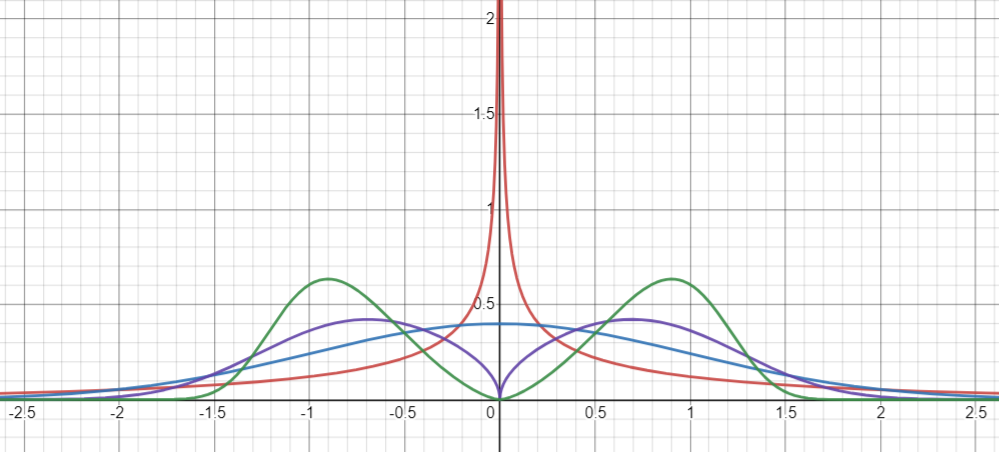} 
\caption{Density function $\varphi_{\alpha}$ depending on the value of parameter $\alpha$.}
\end{figure}
 

\begin{rem}
Starting with the random variable $\sigma G$ $(\sigma >0)$, in the same manner as in Proposition \ref{cdf}
we get that
\begin{eqnarray*} 
\Phi_{\sigma,\alpha}(x)&:=&\mathbb{P}(f_\alpha(\sigma G)\leq x)=\mathbb{P}(\sigma G\leq \operatorname{sgn}(x)|x|^{\alpha/2})\\
\; &=& \Phi\Big(\frac{\operatorname{sgn}(x)|x|^{\alpha/2}}{\sigma}\Big)   
\end{eqnarray*}
and 
$$
\varphi_{\sigma,\alpha}(x):=\Phi_{\sigma,\alpha}^\prime(x)
=\frac{\alpha}{2\sqrt{2\pi}\sigma}|x|^{\alpha/2-1}\exp\Big(-\frac{|x|^{\alpha}}{2\sigma^2}\Big).
$$
\end{rem} 
\begin{rem}
\label{Rad}
Let us observe that for each $\sigma>0$ the distribution function $\Phi_{\sigma,\alpha}=\Phi(\operatorname{sgn}(x)|x|^{\alpha/2}/\sigma)$ tends weakly to the Rademacher distribution as 
$\alpha\to\infty$. For this reason we will denote the Rademacher distribution by $\Phi_\infty$ .

\end{rem}

{\bf The moment problem of $\alpha$-normal distribution}.
Since, for $G\sim\mathcal{N}(0,1)$ and $p>0$,
$$
\mathbb{E}(|G|^{p})=
\frac{2^{p/2}}{\sqrt{\pi}}\Gamma\Big(\frac{p+1}{2}\Big),
$$
we immediately get
$$
\mathbb{E}(|G_\alpha|^{p})=\mathbb{E}(|G|^{2p/\alpha})=
\frac{2^{p/\alpha}}{\sqrt{\pi}}\Gamma\Big(\frac{p}{\alpha}+\frac{1}{2}\Big).
$$
Let us emphasize that $G_\alpha$ and its modulus both have finite integer-order moments.
It is natural to ask about the moment problem (see Stoyanov \cite[Sec.11]{Stoy} for instance). That is, let
$$
m_{\alpha,k}:=\int_{-\infty}^\infty x^k d\Phi_\alpha(x)=\int_{-\infty}^\infty x^k \varphi_\alpha(x)dx,\quad
k=0,1,2,...\;,
$$
we ask whether $\Phi_\alpha$ is uniquely determined ($M$-determinate) or indetermined ($M$-indeterminate) by the sequence of moments $(m_{\alpha,n})$.
Following Stoyanov's reasoning from \cite[Subsection 11.1]{Stoy}, we can answer this question (the necessary definitions and criteria can be found at the beginning of \cite[Section 11]{Stoy}).

Stoyanov considers the random variable $G^3$ in detail, writing that all odd powers of $N(0,1)$ can be considered in a similar way. Note that with our parametrization $G^3=G_{2/3}$ ($\alpha=2/3$) and in general $G^{2n+1}=G_{2/(2n+1)}$ ($\alpha=2/(2n+1)$). Let us emphasize that the presented reasoning is true for any $\alpha>0$. 

By using some standard integrals
($\int_0^\infty\frac{1}{1+x^2}dx=\frac{\pi}{2}$, $\int_0^\infty\frac{\ln x}{1+x^2}dx=0$ and 
$\int_0^\infty \frac{x^\alpha}{1+x^2}dx=\frac{\pi}{2}\cos\frac{\alpha\pi}{2}$, $-1<\alpha<1$)
we conclude that
\begin{eqnarray*}
\int_{-\infty}^\infty\frac{-\ln\varphi_\alpha(x)}{1+x^2}dx &=& -\ln\frac{\alpha}{2\sqrt{2\pi}}\int_{-\infty}^\infty\frac{1}{1+x^2}dx
-\Big[\frac{\alpha}{2}-1\Big]\int_{-\infty}^\infty\frac{\ln|x|}{1+x^2}dx\\
\;&\;&+\frac{1}{2}\int_{-\infty}^\infty\frac{|x|^\alpha}{1+x^2}dx<\infty\quad (0<\alpha<1).
\end{eqnarray*}
Hence, according to the Krein criterion (see ($C_3$) in \cite[Sec.11]{Stoy} for instance)
the distribution of r.v. $\operatorname{sgn}(G)|G|^{2/\alpha}$ is $M$-indeterminate 
for $0<\alpha<1$.

Now we prove that for $\alpha\ge 1$ the d.f. of $G_\alpha$ is $M$-determinate. To simplify the notation we write
$A\sim_d B$ if $\frac{1}{C(d)}A\le B \le C(d)A$, where $C(d)$ is a positive constant, which depends only on $d$.
Let us note that 
$$
m_{\alpha,2k}=\frac{2^{2k/\alpha}}{\sqrt{\pi}}\Gamma\Big(\frac{2k}{\alpha}+\frac{1}{2}\Big)
$$
and 
$$
\Gamma\Big(\frac{2k}{\alpha}+\frac{1}{2}\Big)\le \Gamma\Big(\frac{2k}{\alpha}+1\Big),
$$
for sufficiently large $k$ (e.g. $k\ge \alpha/2$).
By Stirling's formula
$$
\Gamma\Big(\frac{2k}{\alpha}+1\Big)^\frac{1}{2k}\sim_\alpha k^\frac{1}{\alpha}.
$$
Since $\sum_{k=1}^\infty k^{-1/\alpha}=\infty$ if $\alpha\ge 1$, we get that
$$
\sum_{k=1}^\infty m_{\alpha,2k}^{-1/(2k)}=\infty
$$
for $\alpha\ge 1$.
By Carleman's condition (see $(C_2)$ in \cite[Sec.11]{Stoy}) we obtain that $G_\alpha$ is $M$-determinate for 
$\alpha \ge 1$.

Moreover one can calculate that the moment generating function (m.g.f.) of $G_1=\operatorname{sgn}(G)G^2$
takes the form $\frac{1}{2}((1-2t)^{-1/2}+(1+2t)^{-1/2})$ for $t\in(-1/2,1/2)$. Thus we can deduce that $G_\alpha$ 
has \emph{light tails} for $\alpha\ge 1$ (possesses the m.g.f.).
\begin{rem}
Repeating the above reasoning for the random variable $|G_\alpha|=|G|^{2/\alpha}$ and using Krein's and Carleman's conditions  for its probability distribution $2\Phi(t^{\alpha/2})-1$ with the support $[0,\infty)$, one can calculate that $|G_\alpha|$ is $M$-indeterminate for $0<\alpha < 1/2$ and $M$-determinate for
$\alpha \ge 1/2$, although $|G_\alpha|$ has the m.g.f.  only from $\alpha \ge 1$ 
(compare the example $|G|^3=|G_{2/3}|$ in \cite[Sec.11.1]{Stoy}).
\end{rem}

\section{Comparison with the Weibull distribution}
Now we compare the distribution of $G_\alpha$ with the distributions of the Weibull($\alpha,\lambda$) random variables for some $\lambda$'s.
By definition, a random variable $X$ majorizes a random variable $Y$ in distribution, if there exists $t_0\ge 0$ such that 
$$
\mathbb{P}(|X|\ge t)\ge \mathbb{P}(|Y|\ge t),
$$
for any $t>t_0$; see, for instance, \cite[Def. 1.1.2]{BulKoz}.

\begin{pro}
The  $\alpha$-normal random variable $G_\alpha$ majorizes the Weibull($\alpha$,1) random variable and it is majorized by the Weibull($\alpha, 2^{1/\alpha}$) random variable.
\end{pro}
\begin{proof}
It is known that the tails  of the Gaussian random variable can be  estimated from above in  the following way
$$
\mathbb{P}(|G|\ge t)\le \exp(-t^2/2)
$$
for any $t\ge 0$; see, for instance, \cite[Prop.2.2.1]{Dudley}) . Hence for the $\alpha$-normal random variable we get
\begin{equation}
\label{est1}
\mathbb{P}(|G_\alpha|\ge t)=\mathbb{P}(|G|\ge t^{\alpha/2})\le \exp\big(-(t/2^{1/\alpha})^\alpha\big).
\end{equation}
Observe that the right hand side is a tail of the Weibull($\alpha, 2^{1/\alpha}$) random variable. This means that the Weibull($\alpha, 2^{1/\alpha}$) random variable majorizes the $\alpha$-normal random variable.

In the same source \cite[Prop.2.2.1]{Dudley}) one can find the following lower estimate of the tails of the Gaussian random variable
$$
\mathbb{P}(|G|\ge t)\ge\frac{1}{t}\frac{1}{\sqrt{2\pi}}\exp(-t^2/2)
$$
for $t\ge 1$.
Since $\sqrt{2\pi}t\exp(-t^2/2)$ tends to $0$ as $t\to\infty$, there exists $t_0$ such that $\sqrt{2\pi}t\exp(-t^2/2)\le 1$ for $t\ge t_0$. It gives 
$\exp(-t^2/2)\le 1/\sqrt{2\pi}t$ and, in consequence, 
$$
\mathbb{P}(|G|\ge t)\ge \frac{1}{t}\frac{1}{\sqrt{2\pi}}\exp(-t^2/2)\ge\exp(-t^2)  
$$
for $t\ge t_0$. By the above
$$
\mathbb{P}(|G_\alpha|\ge t)=\mathbb{P}(|G|^{2/\alpha}\ge t)=\mathbb{P}(|G|\ge t^{\alpha/2})\ge\exp(-t^\alpha)=\mathbb{P}(W_{\alpha,1}\ge t),
$$
for $t\ge t_0^{2/\alpha}$,
which means that the $\alpha$-normal random variable majorizes the Weibull($\alpha,1$) random variable.
\end{proof}

Although the model  $\alpha$-normal distribution is comparable to the Weibull distribution in the above sense, it is significantly different. We show it using
the entropy function. 
Recall that the differential entropy of the two-parameter Weibull distribution is given by the formula
$$
h(W_{\alpha,\lambda})=\gamma\Big(1-\frac{1}{\alpha}\Big)+\ln\Big(\frac{\lambda}{\alpha}\Big)+1,
$$
where $\gamma$ is the Euler-Mascheroni constant. For $G_\alpha$ this is distinct.
\begin{pro}
The differential entropy of $G_\alpha$ has the following form
$$
h(G_\alpha)=\Big(\frac{1}{2}-\frac{1}{\alpha}\Big)(\gamma+\ln 2)+\ln\frac{2\sqrt{2\pi}}{\alpha}+\frac{1}{2},
$$
where $\gamma$ denotes the Euler-Mascheroni constant.
\end{pro}
\begin{proof}
By the definition of  differential entropy and the form of density $\varphi_\alpha$
 of $G_\alpha$ 
we get
\begin{eqnarray}
\label{Hfa}
h(G_\alpha) & = & -\int_{-\infty}^\infty \varphi_\alpha(x)\ln \varphi_\alpha(x)dx\nonumber\\
\; & = & -\int_{-\infty}^\infty \varphi_\alpha(x)\Big[\ln\frac{\alpha}{2\sqrt{2\pi}}+\Big(\frac{\alpha}{2}-1\Big)\ln|x|-\frac{1}{2}|x|^\alpha\Big]dx\nonumber\\
\; & = & \ln\frac{2\sqrt{2\pi}}{\alpha}+\Big(1-\frac{\alpha}{2}\Big)\int_{-\infty}^\infty \ln|x| \varphi_\alpha(x)dx+\frac{1}{2}\int_{-\infty}^\infty |x|^\alpha \varphi_\alpha(x)dx.
\end{eqnarray}
Let us observe that
\begin{equation}
\label{int1}
\int_{-\infty}^\infty |x|^\alpha \varphi_\alpha(x)dx=\mathbb{E}|G_\alpha|^\alpha=\mathbb{E}G^2=1.
\end{equation}
and 
\begin{equation}
\label{int2}
\int_{-\infty}^\infty \ln|x| \varphi_\alpha(x)dx=\frac{\alpha}{2\sqrt{2\pi}}\int_{-\infty}^\infty \ln|x||x|^{\alpha/2-1}e^{-|x|^{\alpha}/2}dx.
\end{equation}
Substituting $u=x^{\alpha/2}$ ($x>0$) we obtain
\begin{equation}
\label{int3}
\int_{-\infty}^\infty \ln|x||x|^{\alpha/2-1}e^{-|x|^{\alpha}/2}dx=\frac{8}{\alpha^2}\int_0^\infty e^{-\frac{1}{2}u^2}\ln u du.
\end{equation}
By \cite[4.333]{Table} we have that 
\begin{equation}
\label{int4}
\int_0^\infty e^{-\frac{1}{2}u^2}\ln u du=-\frac{1}{4}(\gamma+\ln 2)\sqrt{2\pi}.
\end{equation}
Combining (\ref{int4}), (\ref{int3}), (\ref{int2}), (\ref{int1}) and substituting into (\ref{Hfa}) we obtain
the formula in the assertion.
\end{proof}
For $\alpha=2$ we get the differential entropy $h(G)=\ln\sqrt{2\pi}+1/2$ of the standard Gaussian variable $G$. 

\section{Orlicz norm of the $\alpha$-normal distribution}
\label{secOrlicz}
Let us emphasize that  Weibull random variables are the model examples of random variables with $\alpha$-sub-exponential tail decay. 
We say that a random variable $X$ has the $\alpha$-sub-exponential tail decay if there exist two constant $c,C$ such that for $t\ge 0$ it holds
$$
\mathbb{P}(|X|\ge t)\le c\exp\big(-(t/C)^\alpha\big).
$$
Since 
$$
\mathbb{P}(W_{\alpha,\lambda}\ge t)= \exp\big(-(t/\lambda)^\alpha\big),
$$
the Weibull random $W_{\alpha,\lambda}$ has $\alpha$-sub-exponential tail decay with $c=1$ and $C=\lambda$.
Whereas the estimate (\ref{est1}) means that $G_\alpha$ has such a tail decay with $c=1$ and $C=2^{1/\alpha}$.  

The property of $\alpha$-sub-exponential tail decay can be equivalently expressed in terms of so-called (exponential) Orlicz norms. Recall that for any random variable $X$, $\psi_\alpha$-norm is defined by 
$$
\|X\|_{\psi_\alpha}:=\inf \big\{K>0:\; \mathbb{E}\exp(|X/K|^\alpha)\le 2\big\};
$$
according to the standard convention $\inf\emptyset=\infty$. 
We  call the above functional $\psi_\alpha$-norm but let us emphasize that only for $\alpha\ge 1$ it is a proper norm. For $0<\alpha<1$ it is  quasi-norm. It does not satisfy the triangle inequality (see Appendix A in \cite{GSS} for more details).
One can observe that $\||X|\|_{\psi_\alpha}=\|X\|_{\psi_\alpha}$ and, moreover, one can  check that, for $\alpha,\beta>0$,
$\||X|^\beta\|_{\psi_{\alpha}}=\|X\|^\beta_{\psi_{\alpha\beta}}$; see Lemma 2.3 in \cite{Zaj}.

Since the closed form of the moment generating function of random variable  $G^2$ is known, we can calculate the $\psi_\alpha$-norm of $\alpha$-normal random variable $G_\alpha$.  
Since $G^2$ has $\chi^2_1$-distribution with one degree of freedom whose moment generating function is $\mathbb{E}\exp(sG)=(1-2s)^{-1/2}$ for $s<1/2$, we get 
$$
\mathbb{E}\exp(G^2/K^2)=(1-2/K^2)^{-1/2},
$$
which is less or equal $2$ if $K\ge \sqrt{8/3}$. It gives that $\|G\|_{\psi_2}=\sqrt{8/3}$. The $\psi_2$-norm of $|G|$ is equal to $\psi_2$-norm of $G$. 
By Lemma 2.3 in \cite{Zaj} and the definition of $\alpha$-normal distribution we get
$$
\|G_\alpha\|_{\psi_\alpha}=\||G|^{2/\alpha}\|_{\psi_\alpha}=\|G\|^{2/\alpha}_{\psi_2}=(8/3)^{1/\alpha}.
$$
\begin{rem}
Using the closed form of the moment generating function of the standard exponential random variable and the above mentioned definition of the two-parameter Weibull distribution, similarly as for the standard $\alpha$-Gaussian random variable, one can obtain its $\psi_\alpha$-norm    $\|W_{\alpha,\lambda}\|_{\psi_\alpha}=\lambda 2^{1/\alpha}$.
\end{rem}
\begin{rem}
Although the Weibull($\alpha,\lambda$) random variables provide model examples of random variables with $\alpha$-sub-exponential tail decay (they are model  elements of spaces generated by the $\psi_\alpha$-norms), it can nevertheless be argued that the $\alpha$-Gaussian variables should play a central role among these variables (in these spaces). 
\end{rem}

\section{Multivariate $\alpha$-normal distributions}
We define the multivariate $\alpha$-normal distribution as the meta-Gaussian distribution with $\alpha$-normal margins, i.e., as a composition of the Gauss copula with the $\alpha$-normal distributions. Recall that the Gauss copula $C^{Gauss}_\Sigma$ with the correlation matrix $\Sigma$ is defined as
$$
C^{Gauss}_\Sigma(u_1,...,u_d)={\bf \Phi}_\Sigma(\Phi^{-1}(u_1),...,\Phi^{-1}(u_d)),
$$ 
where ${\bf \Phi}_\Sigma$ is the d.f. of $\mathcal{N}({\bf 0},\Sigma)$ distribution;
see \cite[(5.9)]{MFE}  for instance.

\begin{defn}
We define the {\it joint distribution function of the multivariate $\alpha$-normal distribution} as 
$$
{\bf \Phi}_{\Sigma,\alpha}(x_1,...,x_d):=C^{Gauss}_\Sigma(\Phi_\alpha(x_1),...,\Phi_\alpha(x_d)).
$$
\end{defn}

\begin{pro}
i) The joint distribution function of the multivariate $\alpha$-normal distribution is of the form
$$
{\bf \Phi}_{\Sigma,\alpha}(x_1,...,x_d)={\bf \Phi}_\Sigma\big(\operatorname{sgn}(x_1)|x_1|^{\alpha/2},...,\operatorname{sgn}(x_d)|x_d|^{\alpha/2}\big).
$$
ii) The density function of the multivariate $\alpha$-normal distribution is
$$
\boldsymbol{\varphi}_{\Sigma,\alpha}(x_1,...,x_d)=\Big(\frac{\alpha}{2}\Big)^d\prod_{i=1}^d|x_i|^{\alpha/2-1}\boldsymbol{\varphi}_\Sigma\big(\operatorname{sgn}(x_1)|x_1|^{\alpha/2},...,\operatorname{sgn}(x_d)|x_d|^{\alpha/2}\big),
$$
where $\boldsymbol{\varphi}_\Sigma$ is the density function of $\mathcal{N}({\bf 0},\Sigma)$ distribution. 
\end{pro}
\begin{proof}
By the form of $C^{Gauss}_\Sigma$ and Proposition \ref{cdf} we get forms of the multivariate $\alpha$-normal distribution  and its density function, i.e.,
\begin{eqnarray*}
{\bf \Phi}_{\Sigma,\alpha}(x_1,...,x_d)&=&{\bf \Phi}_\Sigma\Big(\Phi^{-1}\big(\Phi(\operatorname{sgn}(x_1)|x_1|^{\alpha/2}\big),...,\Phi^{-1}\big(\Phi(\operatorname{sgn}(x_d)|x_d|^{\alpha/2}\big)\Big)\\
\; &=&{\bf \Phi}_\Sigma\big(\operatorname{sgn}(x_1)|x_1|^{\alpha/2},...,\operatorname{sgn}(x_d)|x_d|^{\alpha/2}\big)
\end{eqnarray*}
and, for the second part,
\begin{eqnarray*}
\boldsymbol{\varphi}_{\Sigma,\alpha}(x_1,...,x_d)&=&\frac{\partial^d }{\partial x_1\ldots\partial x_d}
{\bf \Phi}_\Sigma\big(\operatorname{sgn}(x_1)|x_1|^{\alpha/2},...,\operatorname{sgn}(x_d)|x_d|^{\alpha/2}\big)\\
&=&\Big(\frac{\alpha}{2}\Big)^d\prod_{i=1}^{d}|x_i|^{\frac{\alpha}{2}-1}\boldsymbol{\varphi}_{\Sigma}\big(\operatorname{sgn}(x_1)|x_1|^{\frac{\alpha}{2}},\dots,\operatorname{sgn}(x_d)|x_d|^{\frac{\alpha}{2}}\big) .
\end{eqnarray*}
\end{proof}
\begin{exa}
Recall that the standard bivariate normal density function $\boldsymbol{\varphi}_\rho$ with the correlation
coefficient $\rho\in(-1,1)$ is of the form
\begin{eqnarray*}
\boldsymbol{\varphi}_\rho(x,y):=\frac{1}{2\pi\sqrt{1-\rho^2}}\exp\Big(-\frac{1}{2(1-\rho)^2}\big[x^2-2\rho xy+y^2\big]\Big).
\end{eqnarray*}
By the above proposition, the  bivariate $\alpha$-normal density function 
$\boldsymbol{\varphi}_{\rho,\alpha}$ with the 
coefficient $\rho$ takes the form
\begin{eqnarray*}
\boldsymbol{\varphi}_{\rho,\alpha}(x,y)&=&\Big(\frac{\alpha}{2}\Big)^2|xy|^{\frac{\alpha}{2}-1}
\boldsymbol{\varphi}_{\rho}\big(\operatorname{sgn}(x)|x|^{\alpha/2},\operatorname{sgn}(y)|y|^{\alpha/2}\big)\\
\;&=&\frac{\alpha^2}{8\pi\sqrt{(1-\rho^2)}}|xy|^{\frac{\alpha}{2}-1}\exp\Big(-\frac{1}{2(1-\rho)^2}\big[|x|^\alpha-2\rho \operatorname{sgn}(xy)|xy|^{\alpha/2}+|y|^\alpha\big]\Big).
\end{eqnarray*}

\end{exa}


\section{Limiting distribution} 
The Gauss copula is a continuous function. Taking into account Remark \ref{Rad} we get the weak convergence of 
the multivariate $\alpha$-normal distribution
to the meta-Gaussian distribution with Rademacher's margins as $\alpha\to\infty$.
This limiting distribution we denote by $\boldsymbol{\Phi}_{\Sigma,\infty}$. Thus
\begin{eqnarray*}
\boldsymbol{\Phi}_{\Sigma,\infty}(x_1,...,x_d) &:=&\lim\limits_{\alpha\to\infty}\boldsymbol{\Phi}_{\Sigma,\alpha}(x_1,...,x_d)=C_{\Sigma}^{Gauss}(\Phi_{\infty}(x_1),\ldots,\Phi_{\infty}(x_d)).
\end{eqnarray*}



Let ${\bf X}$ be a random vector with the cdf $\boldsymbol{\Phi}_{\Sigma,\infty}$. Then 
${\bf X}$ is a discrete random vector with $Ran\; {\bf X}=\{-1,1\}^d$.
We recall the notation for counting elements in lexicographic order. For any vectors $\boldsymbol{x}=(x_j)_{j=1}^{d}$ and $\boldsymbol{y}=(y_j)_{j=1}^{d}$, we denote by $\#\{\boldsymbol{y}\neq \boldsymbol{x}\}$  the number of indices $i$ for which $y_i \neq x_i$, i.e., 
\begin{eqnarray*}
\#\{\boldsymbol{y}\neq \boldsymbol{x}\} = \#\{i: y_i \neq x_i \text{ for } \boldsymbol{x}=(x_j)_{j=1}^{d} \text{ and } \boldsymbol{y}=(y_j)_{j=1}^{d}\}.
\end{eqnarray*}

By the inclusion-exclusion principle and the form of $\boldsymbol{\Phi}_{\Sigma,\infty}$ we immediately get the following form  of the probability mass function of ${\bf X}$.
\begin{pro}
Let $\boldsymbol{X}$ be a random vector with 
the distribution function $\boldsymbol{\Phi}_{\Sigma,\infty}$. Then the probability mass function $P_{\boldsymbol{X}}$ at $\boldsymbol{x}\in\{-1,1\}^d$ is given by
\begin{eqnarray*}
P_{\boldsymbol{X}}(\boldsymbol{x})=\sum_{\boldsymbol{y}:\boldsymbol{y}\preccurlyeq \boldsymbol{x}} (-1)^{\#\{\boldsymbol{y}\neq \boldsymbol{x}\}}\boldsymbol{\Phi}_{\Sigma,\infty}(\boldsymbol{y})
=\sum_{\boldsymbol{y}:\boldsymbol{y}\preccurlyeq \boldsymbol{x}} (-1)^{\#\{\boldsymbol{y}\neq \boldsymbol{x}\}}C_{\Sigma}^{Gauss}(\boldsymbol{\Phi}_{\infty}(\boldsymbol{y})),
\end{eqnarray*}
where $\preccurlyeq$ denotes the lexicographic order on $\{-1,1\}^d$.
\end{pro}

\begin{exa}
Let a random vector $(X,Y)$ has the cdf  $\boldsymbol{\Phi}_{\rho,\infty}$, where $\rho$ is a correlation coefficient of the Gauss copula $C_{\rho}^{Gauss}$. Then
\begin{eqnarray*}
\boldsymbol{\Phi}_{\rho,\infty}(x,y)=C_{\rho}^{Gauss}(\Phi_{\infty}(x),\Phi_{\infty}(y)).
\end{eqnarray*}
Let us note that $\Phi_{\infty}$ takes only three values: $0$, $1/2$ and $1$. By the definition of 
a copula and using its Frechet bounds we have  
$$
C_{\rho}^{Gauss}(1,1)=1,\;
C_{\rho}^{Gauss}(\frac{1}{2},1)=C_{\rho}^{Gauss}(1,\frac{1}{2})=\frac{1}{2}
$$ 
and 
$$
C_{\rho}^{Gauss}(0,0)=C_{\rho}^{Gauss}(0,\frac{1}{2})=C_{\rho}^{Gauss}(\frac{1}{2},0)=0.
$$
The probability mass function of $(X,Y)$ is concentrated at points $(-1,-1)$, $(1,-1)$, $(-1,1)$ and $(1,1)$.
Successively we get
\begin{eqnarray*}
P_{XY}(-1,-1)&=&\mathbb{P}(X=-1,Y=-1)=\boldsymbol{\Phi}_{\rho,\infty}(-1,-1)\\
&=&C_{\rho}^{Gauss}(\Phi_{\infty}(-1),\Phi_{\infty}(-1))=C_{\rho}^{Gauss}\Big(\frac{1}{2},\frac{1}{2}\Big),
\end{eqnarray*}
\begin{eqnarray*}
P_{XY}(1,-1)&=&\mathbb{P}(X=1,Y=-1)=\boldsymbol{\Phi}_{\rho,\infty}(1,-1)-\boldsymbol{\Phi}_{\rho,\infty}(-1,-1)\\&=&C_{\rho}^{Gauss}(\Phi_{\infty}(1),\Phi_{\infty}(-1))-C_{\rho}^{Gauss}(\Phi_{\infty}(-1),\Phi_{\infty}(-1))\\&=&
C_{\rho}^{Gauss}\Big(1,\frac{1}{2}\Big)-C_{\rho}^{Gauss}\Big(\frac{1}{2},\frac{1}{2}\Big)\\
&=&\frac{1}{2}-C_{\rho}^{Gauss}\Big(\frac{1}{2},\frac{1}{2}\Big).
\end{eqnarray*}
Similarly
\begin{eqnarray*}
P_{XY}(-1,1)&=&\mathbb{P}(X=-1,Y=1)=\boldsymbol{\Phi}_{\rho,\infty}(-1,1)-\boldsymbol{\Phi}_{\rho,\infty}(-1,-1)\\&=&C_{\rho}^{Gauss}(\Phi_{\infty}(-1),\Phi_{\infty}(1))-C_{\rho}^{Gauss}(\Phi_{\infty}(-1),\Phi_{\infty}(-1))\\
&=&C_{\rho}^{Gauss}\Big(\frac{1}{2},1\Big)-C_{\rho}^{Gauss}\Big(\frac{1}{2},\frac{1}{2}\Big)\\
&=&\frac{1}{2}-C_{\rho}^{Gauss}\Big(\frac{1}{2},\frac{1}{2}\Big)
\end{eqnarray*}
and finally
\begin{eqnarray*}
P_{XY}(1,1)&=&\mathbb{P}(X=1,Y=1)=\boldsymbol{\Phi}_{\rho,\infty}(1,1)-\boldsymbol{\Phi}_{\rho,\infty}(1,-1)\\
&&-\boldsymbol{\Phi}_{\rho,\infty}(-1,1)+\boldsymbol{\Phi}_{\rho,\infty}(-1,-1)\\
&=&C_{\rho}^{Gauss}(\Phi_{\infty}(1),\Phi_{\infty}(1))-C_{\rho}^{Gauss}(\Phi_{\infty}(1),\Phi_{\infty}(-1))\\
&&-C_{\rho}^{Gauss}(\Phi_{\infty}(-1),\Phi_{\infty}(1))+C_{\rho}^{Gauss}(\Phi_{\infty}(-1),\Phi_{\infty}(-1))\\
&=&C_{\rho}^{Gauss}(1,1)-C_{\rho}^{Gauss}\Big(1,\frac{1}{2}\Big)-C_{\rho}^{Gauss}\Big(\frac{1}{2},1\Big)\\
&&+C_{\rho}^{Gauss}\Big(\frac{1}{2},\frac{1}{2}\Big)=1-\frac{1}{2}-\frac{1}{2}+C_{\rho}^{Gauss}\Big(\frac{1}{2},\frac{1}{2}\Big)\\
&=&C_{\rho}^{Gauss}\Big(\frac{1}{2},\frac{1}{2}\Big)=1-P_{XY}(-1,-1)-P_{XY}(1,-1)-P_{XY}(-1,1).
\end{eqnarray*}
Summarizing
$$
P_{XY}(-1,-1)=P_{XY}(1,1)=C_{\rho}^{Gauss}\Big(\frac{1}{2},\frac{1}{2}\Big)
$$
and
$$
P_{XY}(1,-1)=P_{XY}(-1,1)=\frac{1}{2}-C_{\rho}^{Gauss}\Big(\frac{1}{2},\frac{1}{2}\Big).
$$
\end{exa}

\end{document}